\documentclass[a4paper]{amsart}

\usepackage{a4, amssymb, enumerate, amsmath}
\usepackage{color}

\renewcommand\emptyset\varnothing
\renewcommand\subset\subseteq

\usepackage{ushort}

\DeclareMathOperator{\En}{End}
\DeclareMathOperator{\ran}{ran}
\DeclareMathOperator{\cp}{charPoly}
\DeclareMathOperator{\sign}{sign}
\DeclareMathOperator{\rank}{rank}

\DeclareMathOperator{\ev}{ev}

\DeclareMathOperator{\St}{St}

\usepackage[pagebackref,colorlinks,linkcolor=red,citecolor=blue,urlcolor=blue,hypertexnames=true]{hyperref}
\begin{document}

\newcommand{\la}{\lambda}

\newcommand{\A}{\cal{A}}
\newcommand{\F}{\cal{F}}
\newcommand{\R}{\mathbb{R}}
\newcommand{\cH}{\cal{H}}
\newcommand{\T}{\cal{T}}
\newcommand{\Z}{\cal{Z}}
\newcommand{\J}{\cal{J}}
\newcommand{\cS}{\cal{S}}
\newcommand{\M}{\cal{M}}
\newcommand{\K}{\cal{K}}
\newcommand{\C}{\mathbb{C}}
\newcommand{\U}{\cal{U}}
\newcommand{\B}{\cal{B}}
\newcommand{\V}{\cal{V}}
\newcommand{\cL}{\cal{L}}
\newcommand{\W}{\cal{W}}
\newcommand{\FF}{\mathbb F}
\newcommand{\KK}{\mathbb K}
\newcommand{\Q}{\mathbb Q}
\newcommand{\N}{\mathbb N}
\newcommand{\wh}[1]{\widehat{#1}}
\newcommand{\rsl}{\mathfrak{sl}}
\newcommand{\wt}{\widetilde}
\newcommand{\ov}{\overline}

\newcommand{\fa}{\FF\ax}

\newcommand{\x}{\ushort X}
\newcommand{\ua}{\ushort A}
\newcommand{\us}{\ushort S}
\newcommand{\uaa}{\ushort a}
\newcommand{\ub}{\ushort B}
\newcommand{\ax}{\langle\ushort X\rangle}
\newcommand{\axn}{\langle X_1,\ldots,X_n\rangle}
\newcommand{\axstar}{\langle\ushort X^*\rangle}
\newcommand{\axnstar}{\langle X_1^*,\ldots,X_n^*\rangle}
\newcommand{\cx}{[\ushort X]}
\newcommand{\axs}{\langle\ushort X,\ushort X^*\rangle}
\newcommand{\axns}{\langle X_1,\ldots,X_n,  X_1^*,\ldots,X_n^*\rangle}
\newcommand{\fastar}[1]{{#1}\ax}

\newtheorem{theorem}{Theorem}[section]
\newtheorem{proposition}[theorem]{Proposition}
\newtheorem{thm}[theorem]{Theorem}
\newtheorem{prop}[theorem]{Proposition}
\newtheorem{lemma}[theorem]{Lemma}
\newtheorem{lem}[theorem]{Lemma}
\newtheorem{corollary}[theorem]{Corollary}
\newtheorem{cor}[theorem]{Corollary}

\theoremstyle{definition}
\newtheorem{remark}[theorem]{Remark}
\newtheorem{rem}[theorem]{Remark}
\newtheorem{definition}[theorem]{Definition}
\newtheorem{defn}[theorem]{Definition}
\newtheorem{example}[theorem]{Example}
\newtheorem{ex}[theorem]{Example}
\newtheorem{examples}[theorem]{Examples}
\newtheorem{question}[theorem]{Question}

\newcommand\cal{\mathcal}

\def\beq{\begin{equation} }
\def\eeq{\end{equation} }
\def\beqn{\begin{equation*} }
\def\eeqn{\end{equation*} }
\newcommand{\csim}{\stackrel{\mathrm{cyc}}{\thicksim}}

\def\ben{\begin{enumerate} }
\def\een{\end{enumerate} }
\numberwithin{equation}{section}

\def\bh{\mathcal B(\mathcal H)}
\def\kh{\mathcal K(\mathcal H)}
\def\cA{\mathcal A}
\title[Local-global principle for linear dependence of nc polynomials]{A Local-global principle for linear dependence of noncommutative polynomials}

\author[Matej Bre\v sar and Igor Klep]{Matej Bre\v sar${}^1$ and Igor Klep${}^2$}

\address{Faculty of Mathematics and Physics, University
of Ljubljana, and \\
  Faculty of Natural Sciences and Mathematics, University of Maribor, Slovenia	}
\email{matej.bresar@fmf.uni-lj.si}
\email{igor.klep@fmf.uni-lj.si}
\thanks{${}^1$Supported by the Slovenian Research Agency (Program No. P1-0288).
${}^2$Supported by the Slovenian Research Agency (Project No. J1-3608 and Program No. P1-0222).}
\date{8 March 2011}

\begin{abstract}
A set of polynomials in noncommuting variables is called
{\em locally linearly dependent} if their evaluations
at tuples of matrices are always linearly dependent.
By a theorem of Camino, Helton, Skelton and Ye,
a finite locally linearly dependent set of polynomials
is linearly dependent.
In this short note an alternative proof
based on the 
theory of polynomial identities
is given.
The method of the proof yields
generalizations to {\em directional} local linear dependence
and evaluations in {\em general algebras} over fields of
{\em arbitrary characteristic}.
A main feature of the proof is that it makes it possible
to deduce bounds on the size of the matrices where the 
(directional)
local linear dependence needs to be tested in order to
establish linear dependence.
\end{abstract}

\iffalse
%%%%%%% text version of the abstract
A set of polynomials in noncommuting variables is called locally linearly dependent if their evaluations at tuples of matrices are always linearly dependent. By a theorem of Camino, Helton, Skelton and Ye, a finite locally linearly dependent set of polynomials is linearly dependent. In this short note an alternative proof based on the theory of polynomial identities is given.
The method of the proof yields generalizations to directional local linear dependence and evaluations in general algebras over fields of arbitrary characteristic. A main feature of the proof is that it makes it possible to deduce bounds on the size of the matrices where the (directional) local linear dependence needs to be tested in order to establish linear dependence.
%%%%%%%
\fi

\maketitle
\section{Introduction}\label{sec:intro}

As part of the studies in free analysis
motivated from systems engineering, Camino, Helton,
Skelton and Ye \cite{CHSY}
consider {\em local} linear dependence of functions in noncommuting (nc)
variables, e.g.~polynomials and rational functions.
One of the core results of \cite{CHSY} is that locally linearly dependent
nc polynomials are ({\em globally}) linearly dependent.
This result is what we call the {\em local-global principle for linear
dependence} (of polynomials).
It has been exploited
repeatedly since; often in connection with noncommutative 
convexity and geometry, cf.~\cite{HHLM,DHM,GHV}. We also refer to
the tutorial \cite{HKM} for a more streamlined presentation of the proof and
its applications.

Our aim is to give an algebraists' response to \cite{CHSY}.
That is, we give a proof of this local-global principle that is
 motivated by the theory of polynomial
identities. As such it applies not only to matrix algebras
but to evaluations in general algebras
over fields of arbitrary characteristic. However, even
in the case of matrix algebras it allows us to extract
additional information, e.g.
the size of the matrices
where the local linear dependence needs to be checked in order
to establish linear dependence. Also, we establish bounds in the
case of directional dependence (see below for definitions and
precise statements), something the original proofs do not allow. 

This note is organized as follows. After a preliminary 
Section \ref{sec:prelim}
introducing all the notions needed, we give our main results  in 
Section \ref{sec:main}.
As this is an algebraic paper addressed also to analysts, 
we will give a somewhat detailed treatment of the algebraic tools 
that will be used in our proofs.

\section{Preliminaries}\label{sec:prelim}

\subsection{Notation and set-up} 
In this section we fix the basic notation and terminology we shall use
throughout the paper.
Let  $\FF$ be a field.

\subsubsection{Free algebra} 

By $\FF\ax$ we denote the free associative algebra generated by $\x = \{X_1,X_2,\ldots\}$, i.e., the algebra of all polynomials in noncommuting variables $X_i$. 
We write $\ax$ for the monoid freely
generated by $\x$, i.e., $\ax$ consists of \emph{words} in the
letters $X_1,X_2,\ldots$ (including the empty word denoted by $1$).
Write $\FF\ax_k$ for the vector space consisting of the
polynomials of degree at most $k$.
Sometimes, for notational convenience, we shall also use $Y_j,Z_j$
to denote noncommuting variables.

\subsubsection{Evaluations and representations}

If $p\in\FF\axn$, $\cA$ is an $\FF$-algebra, and $\uaa\in \cA^n$, then
$p(\uaa)\in\cA$ is the {\em evaluation} of $p$ at $\uaa$.
This gives rise to a \emph{representation} {\rm ev}$_{\uaa}: \FF\axn\to\cA$.

\subsubsection{Directional evaluations}
Suppose $\cA$ is a subalgebra of an endomorphism algebra
$\En(V)$ for an $\FF$-vector space $V$.
Given a polynomial $p\in\FF\axn$, an $n$-tuple 
$\uaa\in\cA^n$ and $v\in V$, the expression
$p(\uaa)v$ is called the
{\em directional evaluation} of $p$ in the direction $(\uaa,v)$.

\subsection{Polynomial identities}
We say that $p\in\FF\axn$ is an {\em  identity} of an $\FF$-algebra $\cA$ if $p(\uaa) =0$ for all $\uaa\in \cA^n$. If $p\ne 0$, then $p$ is called a {\em polynomial identity} of $\A$.  For example, $\cA$ is commutative if and only if $\St_2:=X_1X_2-X_2X_1$ is its polynomial identity. We say that $\cA$ is a PI-{\em algebra} if there exists a polynomial identity of $\cA$. Obviously, subalgebras and homomorphic images of PI-algebras are again PI-algebras. Besides commutative algebras, the simplest examples of PI-algebras are finite dimensional ones. To see this, we introduce, for every $n\in\N$, the so-called {\em standard polynomial}  $\St_n = \St_n(X_1,\ldots,X_n)$ by
 \begin{equation} \nonumber
 \St_n := \sum_{\pi \in S_n} \sign(\pi) X_{\pi(1)}\ldots X_{\pi(n)}
 \end{equation}
 where $S_n$ is the symmetric group of degree $n$. It is easy to see that 
 $$\St_n(X_1,\ldots,X_i,\ldots,X_i,\ldots,X_n) = 0,$$
i.e., $\St_n$ vanishes if two variables are the same.
 Accordingly,  $\St_{n}(a_1,\ldots,a_n)=0$ whenever $a_1,\ldots,a_n$ are linearly dependent elements from an algebra $\cA$. This in particular shows that $\St_{n}$ is a polynomial identity of every algebra $\cA$ with $\dim_\FF\cA< n$. Thus, $\St_{d^2+1}$ is a polynomial identity of  the matrix algebra $M_d(\FF)$. There is, however, a much better result, the {\em Amitsur-Levitzki theorem}, saying that  $\St_{2d}$ is a polynomial identity of   $M_d(\FF)$; moreover, it is a polynomial identity of $M_d(\mathcal{C})$ where $\mathcal{C}$ is an arbitrary commutative algebra. The number $2d$ cannot be lowered: a bit tricky, but elementary argument shows that a polynomial of degree $< 2d$ is never a polynomial identity of $M_d(\FF)$. 
%Furthermore,  a polynomial of degree $< 2d$ cannot even vanish on all symmetric $d\times d$ matrices \cite{Sli}.  
 Therefore a polynomial that was a polynomial identity of $M_d(\FF)$ for every $d$ does not exist. This implies that the algebra of all linear operators on an infinite dimensional vector space  (which contains isomorphic copies of all $M_d(\FF)$ as its subalgebras) is not a PI-algebra. In fact, under mild assumptions a PI-algebra is quite close to a matrix algebra. For instance, it turns out that every prime PI-algebra can be embedded into $M_d(\KK)$ for some $d\ge 1$, where $\KK$ is a field extension of the base field $\FF$. Recall that an algebra $\A$ is said to be {\em prime} if the product of any of its two nonzero ideals is nonzero.
 
 Let us also mention a notion related to a polynomial identity: we say that  $p\in\FF\axn$ is a {\em central polynomial} on $M_d(\FF)$ if $p\ne 0$, $p$ is not a polynomial identity, and 
 $p(\ua)$ is a scalar multiple of the identity matrix for every  $\ua\in M_d(\FF)^n$. For instance, $(X_1X_2 - X_2X_1)^2$ is a central polynomial of $M_2(\FF)$, as one can easily check. It is much harder to find central polynomials on $M_d(\FF)$ for larger $d$'s. Anyhow, it is a fact that they do exist for every $d$.  
  
 For full accounts on polynomial identities we refer the reader to \cite{Pro1} and \cite{Row}. 

\subsection{Capelli polynomials and a theorem of Razmyslov}

For $n\in\N$ we define the {\em Capelli polynomial} $C_{2n-1}=C_{2n-1}(X_1,\ldots,X_{2n-1})$ as follows:
\beqn
C_{2n-1}:= \sum_{\pi\in S_n}
\sign(\pi)
 X_{\pi(1)} X_{n+1} 
X_{\pi(2)} X_{n+2} \cdots X_{\pi(n-1)} X_{2n-1} X_{\pi(n)}.
\eeqn
For example, $C_3= X_1X_3X_2-X_2X_3X_1$. Note that by formally replacing all $X_{j}$, $j > n$,  by $1$, $C_{2n-1}$ reduces to $\St_n$.  Just as for the standard polynomials, one can check that 
$$C_{2n-1}(X_1,\ldots,X_i,\ldots,X_i,\ldots,X_n,X_{n+1},\ldots,X_{2n-1}) = 0,$$
implying that for elements $a_i,b_i$ from an algebra $\cA$ we have
$$C_{2n-1}(a_1,\ldots,a_n,b_1,\ldots,b_{n-1}) = 0$$
whenever $a_1,\ldots,a_n$ are linearly dependent. An important theorem of Razmyslov (cf. \cite[Theorem 2.3.7]{BMM} or \cite[Theorem 7.6.16]{Row}) states that the converse of this observation holds in a rather large class of algebras: 

\begin{theorem}\label{thm:raz} 
Let $\A$ be a centrally closed prime algebra. Then $a_1,\ldots,a_n\in \cA$ are linearly dependent if and only if $C_{2n-1}(a_1,\ldots,a_n,b_1,\ldots,b_{n-1}) = 0$ for all $b_j\in \A$. 
\end{theorem}

 The definition of a centrally closed prime algebra is too technical to be included here. The reader is referred to \cite{BMM} for a detailed, or to \cite{BCM} for an informal survey on this notion.
 Let us just mention what is relevant for our applications of Theorem \ref{thm:raz}: the  free algebra $\FF\ax$ is a centrally closed prime algebra (cf. \cite[Theorem 2.4.4]{BMM}). 

Let us conclude this section by mentioning that we have used
Razmyslov's Theorem \ref{thm:raz} before -- in \cite[Section 5.5]{BK} 
to prove
a tracial multilinear Nullstellensatz.

\subsection{Locally linearly dependent operators} 
 
 Let $U$ and $V$ be vector spaces over $\FF$. We say that linear operators 
$T_1, \ldots , T_m : U \to V$ are {\em locally linearly dependent} if $T_1u, \ldots , T_m u$ are linearly dependent vectors in $V$ for every $u\in U$. This does not necessarily mean that $T_1, \ldots , T_m $ are linearly dependent operators. Say, if $T_1$ and $T_2$ are rank one operators with the same range, then they are obviously locally linearly dependent, but not necessarily linearly dependent.
Another example: if $\dim_F V < m$, then any linear  operators 
$T_1, \ldots , T_m : U \to V$ are locally linearly dependent, but there is no reason to believe that they are linearly dependent.
 
 The following result shows that the local linear dependence is intimately connected with the finite rank condition.
 
\begin{thm}\label{thm:bs}
If 
$T_1, \ldots , T_m : U \to V$ are locally  linearly dependent operators, then 
then there exist $\alpha_1,\ldots ,\alpha_m \in \FF$, not all zero, such that $S = \alpha_1 T_1 + \cdots + \alpha_mT_m$ satisfies
%\beq\label{eq:bs}
$	\rank S\leq m - 1.$
This inequality  is sharp.
\end{thm}

It seems that the first result of this kind was obtained by Amitsur \cite{Ami}, however, with a  conclusion that $	\rank S\leq {m+1\choose 2}-1$.
For $\FF = \mathbb C$ Theorem  \ref{thm:bs} was proved by 
Aupetit  \cite{Aup}, for $\FF$ an infinite field by Bre\v sar and \v Semrl \cite{BS}, and finally for $\FF$ a finite field by  Meshulam and \v Semrl \cite{MS}.
The form in which we shall apply Theorem \ref{thm:bs} is as follows:

\begin{lem}\label{lem:bs}
Keep the assumptions of Theorem {\rm\ref{thm:bs}} and assume $U=V$.
 Then the
rank of $C_{2m-1}(T_1,\ldots,T_m,D_1,\ldots,D_{m-1})$
cannot exceed $(m-1) m!$ for any linear $D_j:U\to U$.
\end{lem}

\begin{proof}
By Theorem \ref{thm:bs} we may without loss of generality assume
$T_m$ is a linear combination of $T_1,\ldots,T_{m-1}$ plus an operator
 $E$ of rank $\leq m-1$, e.g.,
\beqn
T_m=E+\sum_{j=1}^{m-1} \alpha_j T_j
\eeqn
for some scalars $\alpha_j$.
Then
\begin{multline*}
C_{2m-1}(T_1,\ldots,T_{m-1},T_m,D_1,\ldots,D_{m-1}) \\
= 
C_{2m-1}\Big(T_1,\ldots,T_{m-1},E+\sum_{j=1}^{m-1} \alpha_j T_j,D_1,\ldots,D_{m-1}\Big).
% \\
%=&
%C_{2m-1}(T_1,\ldots,T_{m-1},\sum_{j=1}^{m-1} \alpha_j T_j,D_1,\ldots,D_{m-1})
%+ (*),
\end{multline*}
Since $C_{2m-1}$ is linear in each variable, it follows that 
\beq\label{eq:rhs}
C_{2m-1}(T_1,\ldots,T_{m-1},T_m,D_1,\ldots,D_{m-1}) = C_{2m-1}(T_1,\ldots,T_{m-1},E,D_1,\ldots,D_{m-1}).
\eeq
The right-hand side of \eqref{eq:rhs} is equal to a sum of $m!$ operators, each of which has rank $\le m-1$. This yields the desired bound.
\end{proof}

\begin{remark} \label{rem}
Let us add that if $A:U\to U$ is a linear
operator of rank $r$, then 
$A,A,\ldots,A^{r+1}$ are linearly
dependent. This is well-known, but let us give a short proof for the sake of completeness. 
Consider the induced mapping $$\check A:U/\ker(A)\to\ran(A).$$
Since $\dim \big(U/\ker(A)\big)=\dim (\ran(A))= r$,
the characteristic polynomial 
\beq
\cp_{\check A}= a_0 + a_1 \lambda + \cdots + (-1)^r \lambda^r \in\FF[\lambda]
\eeq
is of degree $r$ and kills $\check A$ by the Cayley-Hamilton theorem.
Thus
\beq
a_0 A + a_1 A^2 + \cdots + (-1)^{r} A^{r+1}=0.\qedhere
\eeq
\end{remark}

\subsection{Local (directional) linear dependence of polynomials}

Let $\cA$ be an $\FF$-algebra and 
let $S\subseteq \FF\ax$.

\ben[\rm(1)]
\item
We say that $S$ is 
{\em $\cA-$locally linearly dependent}
if the elements 
$$\{ p(\ua) \mid p\in S\}\subset \cA$$
are linearly dependent
for every $\ua\in\cA^\N$.
\item
Now suppose $\cA$ is a subalgebra of 
$\En(V)$ for an $\FF$-vector space $V$.
We say that $S$ is {\em $\cA$-locally directionally linearly
dependent}
if the vectors
$$
\{ p(\ua)v \mid p\in S\}\subset V$$
are linearly dependent
for every $\ua\in\cA^\N$ and $v\in V$.
\item
$S$ is (globally) {\em linearly dependent} if it is linearly dependent
in $\FF\ax$, i.e., there are $\alpha_s\in\FF$ ($s\in S)$, of which 
finitely many are nonzero but not all are zero, such that
$$
0 = \sum_{s\in S} \alpha_s s.
$$
\een

Our core example is $\cA=M_d(\FF)$, but our methods allow us to 
consider evaluations in general algebras.
For instance, in 
Section \ref{sec:nonPI}
we establish that a finite set $S$ of polynomials is $\cA-$locally
linearly dependent for a non-PI algebra $\cA$ if and only if $S$ is
linearly dependent.
In fact, the same conclusion holds whenever $S$ is $M_d(\FF)-$locally 
(directionally) linearly dependent for some $d$ large enough.
As a side product of our proofs we establish bounds on  $d$.

\begin{rem}
The notion of local (directional) linear dependence in free algebras is nontrivial.
For example, if $\cA$ is an $n$-dimensional algebra and $S\subseteq \FF\ax$ is any set with $|S| > n$, then 
 $S$ is 
 $\cA-$locally linearly dependent. Another example:
 any two central polynomials for $d\times d$ matrices
are $M_d(\FF)-$locally linearly dependent, although they need not be
linearly dependent in the free algebra. 
\end{rem}

\section{Results}\label{sec:main}

%This section contains our main results and their proofs.

\subsection{Local linear dependence}
\label{sec:nonPI}
We begin with one of the two of our main results.

\begin{thm}\label{thm:main}
Let $\cA$ be an $\FF-$algebra and let  $f_1,\ldots,f_m\in\FF\axn$ be $\cA-$locally linearly dependent. If $\cA$ does not satisfy a polynomial identity of degree 
\begin{equation}\label{em}
\beta:=\sum_j \deg(f_j)+m-1,
\end{equation}
 then 
$f_1,\ldots,f_m$ are linearly dependent.
\end{thm}

\begin{proof}
For all $\uaa\in \cA^n$, the elements
$f_1(\uaa),\ldots, f_m(\uaa)$ are linearly dependent. Hence 
\beq\nonumber
C_{2m-1}(f_1(\uaa),\ldots, f_m(\uaa),b_1,\ldots,b_{m-1})=0
\eeq
for all $b_j\in\cA$. That is, $C_{2m-1}(f_1,\ldots,f_m,Y_{1},\ldots,Y_{m-1})$ is an identity of $\A$. Since the degree of this polynomial is $\sum_j \deg(f_j)+m-1$, it follows from our assumption that
$
C_{2m-1}(f_1,\ldots,f_m,Y_{1},\ldots,Y_{m-1}) = 0.
$
As the $f_i$'s do not depend on $Y_{1},\ldots,Y_{m-1}$, this trivially yields an apparently stronger conclusion
\beq\nonumber
C_{2m-1}(f_1,\ldots,f_m,h_{1},\ldots,h_{m-1}) = 0
\eeq
for all $h_{1},\ldots,h_{m-1} \in \FF\ax$.
 Hence by  Theorem \ref{thm:raz}, applied to the algebra $ \FF\ax$,
 $f_1,\ldots, f_m$ are linearly dependent.
\end{proof}

The bound $\beta$ may be sometimes too big, but in general it cannot be improved.

\begin{example}
 Let $f_1$ be a polynomial identity of $\cA$ of minimal degree. The set $\{f_1\}$ is then linearly independent and $\cA-$locally linearly dependent. In this case $\beta = \deg{f_1}$, so that $\cA$ satisfies a polynomial identity of degree $\beta$, and does not satisfy a polynomial identity of degree $ < \beta$.
\end{example}

\begin{example}
Let $\cA = \FF$. Let $f_1 = X_1$ and $f_2 = 1$. Obviously,
 the set $\{f_1,f_2\}$ is linearly independent and $\cA-$locally linearly dependent,  $\beta = 2$, $\cA$ satisfies a polynomial identity of degree $2$, and does not satisfy a polynomial identity of degree $ < 2$.
\end{example}

\begin{cor}\label{t1}
Let $\cA$ be a non-PI algebra and let $S\subset\FF\ax$ be an
$\cA-$locally linearly dependent finite set of polynomials.
Then $S$ is linearly dependent.
\end{cor}

Since $M_s(\FF)$ does not satisfy a polynomial identity of degree $< 2s$, we also have the following corollary.

\begin{cor}\label{cor:v1a}
Let $f_1,\ldots,f_m\in\FF\axn$ be $M_s(\FF)-$locally linearly dependent
for some 
\beq\label{eq:bound1}
s> \frac12 \big(\sum_j \deg(f_j)+m-1\big). 
\eeq
Then
$f_1,\ldots,f_m$ are linearly dependent.
\end{cor}

\begin{ex}
Infinitary versions of Corollary \ref{cor:v1a} fail.
\ben[\rm(1)]
\item
Since  $\St_{2n}$ is a polynomial identity for $M_n(\FF)$,
 the set 
\beq\label{eq:ex1}
S=\{\St_{2n}\mid n\in\N\}
\eeq
 is $M_s(\FF)-$locally linearly dependent (for every $s\in\N$),
but is obviously not linearly dependent in $\FF\ax$.
\item
To obtain an infinite set of polynomials in a {\em bounded} number
of variables that is locally but not globally linearly dependent
one just uses the set \eqref{eq:ex1} together with
 the fact that the free algebra $\FF\ax$ embeds
into the free algebra $\FF\langle X,Y\rangle$ on two variables
\cite[Section 2.5, Exercise 18]{Coh}
via
$$
\iota(X_1)=X, \quad \iota(X_2)=[X,Y], \ldots,
\quad 
\iota(X_n)= \big[ \iota(X_{n-1}),Y \big] ,\ldots.
$$
\een
\end{ex}

\subsection{Local directional linear dependence}\label{sec:34}

We now turn to (local) directional linear dependence.
The conclusion here is again that a finite set
of locally directionally linearly dependent polynomials
is indeed linearly dependent. However, the proof is somewhat
more involved and the bounds obtained are worse.

\begin{thm}\label{thm:v2a}
Let $\cA$ be an algebra of linear operators. 
If $f_1,\ldots,f_m\in\FF\axn$ are $\cA-$locally directionally linearly dependent and $\cA$ does not satisfy a polynomial identity of degree
 
\beq\label{eq:bound2}
\gamma := 
\frac{(d+1)(d+2)}2 \big(m-1+\sum_j \deg(f_j) \big) + d,
\quad \text{where } d=(m-1)m!,\eeq
 then
$f_1,\ldots,f_m$ are linearly dependent.
\end{thm}

\begin{proof}
Let $V$ be the space on which operators from $\cA$ act.
Choose $A_i, B_j\in\cA$, $1\le i\le n$, $i\le j\le m-1$.
Let us consider 
$$
H:=C_{2m-1}(f_1(\ua),\ldots, f_m(\ua),B_1,\ldots,B_{m-1})\in\cA.
$$
By assumption, for each 
$v\in V$ the vectors $f_1(\ua)v,\ldots,f_m(\ua)v$ are linearly dependent.
Hence by Lemma \ref{lem:bs},
the rank of $H$ is at most $(m-1)m!=:d$.
So $H,H^2,\ldots, H^{d+1}$ are linearly dependent (cf. Remark \ref{rem}). In particular,
\beq
C_{2d+1}(H,H^2,\ldots,H^{d+1},D_1,\ldots,D_d)=0
\eeq
for all  $D_j\in M_s(\FF)$.
This shows that the polynomial
\beq
g=C_{2d+1}
(h,h^2,\ldots, h^{d+1},
Z_1,\ldots,Z_d)
,
\eeq
where 
\beq 
h=C_{2m-1}(f_1,\ldots,f_m,Y_1,\ldots, Y_{m-1})
\eeq
and $Y_j,Z_j$ are noncommuting indeterminates, is an identity of $\cA$.
Its degree is
\beq
\frac{(d+1)(d+2)}2 \deg(h)+d = 
\frac{(d+1)(d+2)}2 \big(m-1+\sum_j \deg(f_j) \big) + d.
\eeq
According to our assumption this implies  
$g=0$. Repeating the argument based on Theorem \ref{thm:raz}  from the proof of Theorem \ref{thm:main}
we see that
$h,h^2,\ldots, h^{d+1}$ are linearly
dependent polynomials. By comparing degrees in the $Y_j$, 
this is only possible if
$ h=0.$
Applying Theorem \ref{thm:raz} again, we obtain
that $f_1,\ldots,f_m$ are linearly dependent. 
\end{proof}

We explicitly state the matrix version of Theorem \ref{thm:v2a}:

\begin{cor}\label{cor:v2a}
Let $f_1,\ldots,f_m\in\FF\axn$ be $M_s(\FF)-$locally linearly dependent
for some 
\beq\label{eq:bound3}
s>\frac{(d+1)(d+2)}4 \big(m-1+\sum_j \deg(f_j) \big) + \frac d2,
\quad \text{where } d=(m-1)m!,
\eeq
Then
$f_1,\ldots,f_m$ are linearly dependent.
\end{cor}

\subsection{The Fock alternative}

A representation theoretic proof (with some functional analytic flavor)
of the local-global principles (over matrix algebras) 
can be given using the noncommutative
Fock space. Note that the bounds obtained in this way are
different from those given above in that they do not depend
on the number of polynomials under consideration, but
do depend on the number of variables appearing in our polynomials.

\begin{prop}\label{prop:fock}
Suppose 
$f_1,\ldots,f_m\in\FF\axn$ be $M_s(\FF)-$locally directionally
linearly dependent
for 
\[
\begin{split}
s& \geq \dim\R\axn_k  =
\sum_{i=0}^k n^i % = \frac{n^{k+1}-1}{n-1}
=:\sigma,
\end{split}
\]
where $k:=\max\{ \deg(f_j) \mid j=1,\ldots, m\}$.
Then $f_1,\ldots, f_m$ are linearly dependent.
\end{prop}

\begin{proof}
The proof is based on the noncommutative Fock space.
  Define linear operators
$S_j$ on $\R\axn_k$
by declaring, for a word 
  $v\in\axn_k$: 
\beq\label{eq:creation}
   S_j v = \begin{cases} X_j v & \deg(v)<k \\
0 & \text{otherwise}.
\end{cases}
\eeq
By construction, if $p\in\R\ax_k$,
then
\beq\label{eq:fockProp}
p(\us) 1 = p.
\eeq 
({\em Note:} The evaluation $\ev_{\us}$ yields a homomorphism
$
\R\axn \to M_{\sigma}(\FF)
$
which is one-to-one when restricted to $\R\axn_k$.)

Now if $f_1,\ldots, f_m$ are $M_s(\FF)-$locally (directionally) linearly
dependent, then by considering the directional
evaluation at $(\us,1)$, there exist
scalars $\alpha_m$ not all zero satisfying
$$
0 = \sum_{j=1}^m \alpha_m f_m(\us) 1  = \sum_{j=1}^m \alpha_m f_m.
$$
Hence the $f_j$ are linearly dependent.
\end{proof}

\subsection{Free algebras with involution}\label{subsubsec:invo}

Often one is interested in evaluating polynomials
at tuples of symmetric matrices, or is considering
polynomials in disjoint tuples of variables 
$\x,\x^*$ with the obvious notion of evaluation.
In these cases one considers one of the two notions
of free algebras with involution (symmetric
variables or free variables).
All of the results given above have corresponding 
adaptations to free algebras with involution.
The easy verifications are left as an exercise for the reader;
the only nontrivial modifications are the
results needed in the proofs. For instance,
by \cite{Sli},
a polynomial of degree $< 2d$ cannot vanish on all symmetric $d\times d$ matrices.


\begin{thebibliography}{CHMN}

\bibitem[Ami]{Ami}
S.\,A. Amitsur, Generalized polynomial identities and pivotal monomials, 
{\em Trans. Amer. Math. Soc.} {\bf 114} (1965) 210--226.

\bibitem[Aup]{Aup}
B. Aupetit, {\em A primer on spectral theory}, 
Universitext, Springer, 1991.

\bibitem[BMM]{BMM}
K.\,I. Beidar, W.\,S. Martindale 3rd, A.\,V. Mikhalev, {\em Rings
with generalized identities}, Marcel Dekker, Inc., 1996.

\bibitem[BCM]{BCM}
M. Bre\v sar, M.\,A. Chebotar, W.\,S. Martindale,   {\em
Functional identities},
Birkh\"auser, 2007.

\bibitem[BK]{BK}
M. Bre\v sar, I. Klep,
Tracial Nullstellens\"atze, 
accepted for publication in the {\em Borcea memorial volume}, Birkh\" auser, 2011.

\bibitem[B\v S]{BS} M. Bre\v sar, P. \v Semrl,
On locally linearly dependent operators and derivations, 
{\em Trans. Amer. Math. Soc.} {\bf 351} (1999) 1257--1275.

\bibitem[CHSY]{CHSY}
J.\,F. Camino, J.\,W. Helton, R.\,E. Skelton, J. Ye,
Matrix inequalities: a symbolic procedure to determine convexity automatically, 
{\em Integral Equations Operator Theory} {\bf 46} (2003) 399--454.

\bibitem[Coh]{Coh} 
P.\,M. Cohn, 
{\it Free ideal rings and localization in general rings},
Cambridge University Press, 2006.

 \bibitem[DHM]{DHM} H. Dym, J.\,W. Helton, S. McCullough:
The Hessian of a non-commutative polynomial has numerous
  negative  eigenvalues, 
{\em J. Anal. Math.} {\bf 102}    (2007) 29--76.

\bibitem[HHLM]{HHLM}
D.\,M. Hay, J.\,W. Helton, A. Lim, S. McCullough:
Non-commutative partial matrix convexity,
{\em Indiana Univ. Math. J.} {\bf 57} (2008) 2815-–2842.

\bibitem[GHV]{GHV}
J.\,M. Greene, J.\.W. Helton, V. Vinnikov,
Noncommutative plurisubharmonic polynomials Part I: global assumptions,
{\em preprint} \url{http://arxiv.org/abs/1101.0107}

\bibitem[HKM]{HKM}
J.\,W. Helton, I. Klep, S. McCullough,
Tutorial on noncommutative  convex algebraic geometry,
to appear in: 
{\em Semidefinite optimization and convex algebraic geometry},
edited by B. Sturmfels et al.

\bibitem[M\v S]{MS}
R. Meshulam, P. \v Semrl, 
Locally linearly dependent operators,
{\em Pacific J. Math.} {\bf 203} (2002) 441--459. 

\bibitem[Pro]{Pro1} C. Procesi,
{\em Rings with polynomial identities},
Marcel Dekker, Inc., 1973.

\bibitem[Row]{Row}
L.\,H. Rowen, 
{\em Polynomial identities in ring theory}, Academic Press, 1980.

\bibitem[Sli]{Sli} A.\,M. Sli'nko, Special varieties of Jordan algebras, {\em Mat. Zametki} {\bf 26} (1979) 337--344.

\end{thebibliography}
\end{document}